\documentclass[twoside,11pt,reqno]{amsart}
\usepackage{amsmath,amssymb,amscd,mathrsfs,amscd}
\usepackage{graphics,verbatim}
\usepackage{enumitem}
\usepackage{hyperref}
\usepackage{graphicx}

\makeatletter
\newcommand*\bigcdot{\mathpalette\bigcdot@{.5}}
\newcommand*\bigcdot@[2]{\mathbin{\vcenter{\hbox{\scalebox{#2}{$\m@th#1\bullet$}}}}}
\makeatother

\newcommand{\Dist}[1]{\operatorname{Dist}(#1)}
\newcommand{\Distp}[1]{\operatorname{Dist}^+(#1)}

\makeatletter
\newcommand{\leqnomode}{\tagsleft@true}
\newcommand{\reqnomode}{\tagsleft@false}
\makeatother

\newtheorem{theorem}{Theorem}[subsection]

\makeatletter\let\c@fact\c@theorem\makeatother

\makeatletter\let\c@note\c@theorem\makeatother

\newtheorem{lemma}{Lemma}[subsection]
\makeatletter\let\c@lemma\c@theorem\makeatother

\makeatletter\let\c@lemma\c@theorem\makeatother

\newtheorem{quest}{Question}[subsection]
\makeatletter\let\c@alg\c@theorem\makeatother

\newtheorem{prop}{Proposition}[subsection]
\makeatletter\let\c@prop\c@theorem\makeatother

\makeatletter\let\c@conj\c@theorem\makeatother

\makeatletter\let\c@cor\c@theorem\makeatother

\makeatletter\let\c@defn\c@theorem\makeatother

\theoremstyle{definition}

\makeatletter\let\c@remark\c@theorem\makeatother

\makeatletter\let\c@example\c@theorem\makeatother
\numberwithin{equation}{subsection}

\usepackage[capitalise]{cleveref}
\crefname{theorem}{Theorem}{Theorems}
\crefname{fact}{Fact}{Facts}
\crefname{note}{Note}{Notes}
\crefname{lemma}{Lemma}{Lemmas}
\crefname{alg}{Algorithm}{Algorithms}
\crefname{remark}{Remark}{Remarks}
\crefname{example}{Example}{Examples}
\crefname{prop}{Proposition}{Propositions}
\crefname{conj}{Conjecture}{Conjectures}
\crefname{cor}{Corollary}{Corollaries}
\crefname{defn}{Definition}{Definitions}
\crefname{equation}{\!\!}{\!\!} 

\newcounter{listequation}

\begin{document}

\title{The Torus Centralizing Subalgebra of $\Dist{G_r}$}

\author{Paul Sobaje}
\address{Department of Mathematics \\
          Georgia Southern University}
\email{psobaje@georgiasouthern.edu}
\date{\today}

\begin{abstract}
Let $G$ be a simple and simply connected algebraic group over a field of characteristic $p>0$, and $G_r$ its $r$-th Frobenius kernel.  In this paper, we initiate a general study of $\Dist{G_r}^T$, the subalgebra of $\Dist{G_r}$ consisting of fixed points for the adjoint action of a maximal torus $T$ of $G$.  We analyze the structure of this algebra, and classify its simple modules, which essentially are just the non-zero weight spaces of the simple $G_rT$-modules of $p^r$-restricted highest weight.  Further connections between the representations of $\Dist{G_r}^T$ and $G_rT$ are shown, demonstrating the potential usefulness of this algebra.
\end{abstract}

\maketitle


\section{Introduction}

\subsection{Setting}

Let $G$ be a simple and simply-connected algebraic group over an algebraically closed field $\Bbbk$.  Much of the representation theory of $G$ is captured by a maximal torus $T \le G$.  Indeed, when $\text{char}(\Bbbk)=0$, the iso-type of a finite dimensional $G$-module is determined by its iso-type over $T$.  In positive characteristic, the relationship between $G$-modules and $T$-modules is more ambiguous.  Nonetheless, weights and character formulas still occupy a central part in the theory (for example \cite{RW}, \cite{So}) .

Let $N_G(T)$ (resp., $C_G(T)$) denote the normalizer (resp., centralizer) of $T$ in $G$.  In view of the important role that $T$ plays in the study of $G$-modules, it is worth considering the actions of these larger subgroups on a finite-dimensional $G$-module $M$.  In the case of the normalizer, this leads to the familiar fact that the weights of $M$ appear in $W$-orbits, where $W = N_G(T)/T$ is the Weyl group of $G$.  It turns out, however, that nothing new is gained from the action of $C_G(T)$ because $C_G(T)=T$.

In this paper we look beyond the group $G$ to centralizing subalgebras of $T$ inside of $\Dist{G}$, the algebra of distributions on $G$.  The algebra $\Dist{G}$ is also known as the ``hyperalgebra of $G$,'' and in characteristic $0$ it is isomorphic to the universal enveloping algebra of $\text{Lie}(G)$.  The conjugation action of $G$ on itself induces an adjoint action of $G$ on $\Dist{G}$, so it is more accurate to say we will be studying subalgebras of $\Dist{G}^T$, the subalgebra of elements that are fixed by $T$ under the adjoint action.  These elements ``commute with $T$'' in the sense that if $M$ is a finite-dimensional $G$-module, then for all $m \in $M, $x \in \Dist{G}^T$, and $t \in T$, we have
$$x.(t.m)=t.(x.m).$$

Our interest is in the case that $\text{char}(\Bbbk)=p>0$.  Let $G_r $ denote the $r$-th Frobenius kernel of $G$.  The distribution algebra of $G$ is a union of the distribution algebras of the Frobenius kernels, so that
$$\Dist{G} = \bigcup_{r \ge 1} \Dist{G_r}.$$
Each $G_r$ is normal in $G$, so $\Dist{G_r}$ is a $T$-algebra under the adjoint action, and we have
$$\Dist{G}^T = \bigcup_{r \ge 1} \Dist{G_r}^T.$$
Our aim then is to study the subalgebras $\Dist{G_r}^T$.

\subsection{Results}

To the best of our knowledge, an in-depth study of these algebras has only been carried out by Yoshii \cite{Y2}, who looked at the special case of $G=SL_2$.  There, $\Dist{G_r}^T$ (denoted as $\mathcal{A}_r$) was referred to as the ``subalgebra of degree 0,'' and in that special case it is a commutative algebra (this fails to hold for groups of larger rank).  The main result was a description of the projective indecomposable $\Dist{G_r}^T$-modules.  This result relied on explicit computations of primitive idempotents in $\Dist{G_r}$ by Seligman \cite{Se} for $r=1$ and by Yoshii \cite{Y1} for arbitrary $r$.  It turns out that a complete set of orthogonal idempotents for $\Dist{G_r}$ is contained inside of $\Dist{G_r}^T$.

In the present paper, we initiate a general treatment of the algebraic structure and representation theory of $\Dist{G_r}^T$.  We show that, as for $SL_2$, a complete set of orthogonal primitive idempotents for $\Dist{G_r}$ can be found in $\Dist{G_r}^T$, and we give a complete classification of the simple $\Dist{G_r}^T$-modules.  In this generality, explicit computation of the primitive idempotents in $\Dist{G_r}$ is not feasible, so we instead deduce facts about $\Dist{G_r}^T$-modules by restriction from the representation theory of $G_rT$.  In fact, our main motivation behind this work is a better understanding of multiplicities in $G_rT$-modules, and we outline in Subsection \ref{SS:future} how it is possible to compute $G_rT$-composition multiplicities from those in $\Dist{G_r}^T$-modules.

\section{Preliminaries}

For the reader's convenience we briefly recall properties and notation for distribution algebras, all of which can be found in greater detail in \cite[II.1-3]{rags}.

\subsection{Bases For The Distribution Algebra}

Let $\Bbbk[G]$ denote the coordinate algebra of $G$.  It is a finitely generated commutative Hopf algebra over $\Bbbk$.  The distribution algebra $\Dist{G}$ is a subalgebra of the dual space $\Bbbk[G]^*$, and is a cocommutative Hopf algebra over $\Bbbk$.  In characteristic $0$, $\Dist{G}$ is also finitely generated as a $\Bbbk$-algebra, and is isomorphic to the universal enveloping algebra of $\text{Lie}(G)$.  In positive characteristic, $\Dist{G}$ is not finitely generated as a $\Bbbk$-algebra, and is distinct from the universal enveloping algebra of $\text{Lie}(G)$.

The set of roots (resp. positive, negative) for $T$ is $\Phi$ (resp. $\Phi^+$, $\Phi^-$).  We fix $B \le G$ to be the Borel subgroup containing $T$ and the root subgroups corresponding to $\Phi^-$.  The opposite Borel subgroup to $B$ is denoted $B^+$.  Let $U, U^+$ denote the respective unipotent radicals.  Inclusions of algebraic subgroups and of affine subgroup schemes induce embeddings on the level of their distribution algebras.

Subgroup inclusion, followed by multiplication, yields variety isomorphisms
$$B \cong T \times U \cong U \times T.$$
From this we have that the inclusions of subalgebras, followed by multiplication in $\Dist{B}$, yield isomorphisms
$$\Dist{B} \cong \Dist{T} \otimes \Dist{U} \cong \Dist{U} \otimes \Dist{T}.$$
In a similar way, $\Dist{G}$ is equal to the distribution algebra of the ``big cell'' $BU^+$, and we have vector space isomorphisms
\begin{equation}
\Dist{G} \cong \Dist{U} \otimes \Dist{T} \otimes \Dist{U^+} \cong \Dist{T} \otimes \Dist{U} \otimes \Dist{U^+}.
\end{equation}

Let $\alpha \in \Phi$.  There is a root subgroup $U_{\alpha}$ whose coordinate algebra $\Bbbk[U_{\alpha}]$ is a polynomial ring in the variable $Y_{\alpha}$.  For each $n \ge 1$, let $X_{\alpha}^{(n)}$ denote the linear functional that is dual to $Y_{\alpha}^n$.  We will also write $X_{\alpha}$ for $X_{\alpha}^{(1)}$.  Setting $X_{\alpha}^{(0)}=1$, we have that that the collection
$$\{X_{\alpha}^{(n)}\}_{n \ge 0} = \{1, X_{\alpha}, X_{\alpha}^{(2)}, X_{\alpha}^{(3)}, \ldots\}$$
is a basis for $\Dist{U_{\alpha}}$.

The isomorphisms above can be further expanded as an isomorphism \cite[II.1.12(2)]{rags}
$$\Dist{T} \otimes \left( \bigotimes_{\alpha \in \Phi^+} \Dist{U_{-\alpha}} \right) \otimes  \left( \bigotimes_{\alpha \in \Phi^+} \Dist{U_{\alpha}} \right) \xrightarrow{\sim} \Dist{G}$$
that arises from inclusion followed by multiplication.  Let $\alpha_1,\alpha_2,\ldots, \alpha_s$ be an ordering of all the elements in $\Phi^+$, and let $n_1,n_2,\ldots,n_s$ and $m_1,m_2,\ldots, m_s$ be nonnegative integers.  One basis for $\Dist{G}$ is given by all elements of the form
\begin{equation}\label{E:totalweight}
H \cdot X_{-\alpha_1}^{(m_1)}\cdots X_{-\alpha_s}^{(m_s)} X_{\alpha_1}^{(n_1)}\cdots X_{\alpha_s}^{(n_s)}
\end{equation}
where $H$ is an element from a fixed basis for $\Dist{T}$.  More detailed descriptions (as in \cite[II.1.12]{rags}) can be given by describing a basis for $\Dist{T}$, but for the purposes of this work are not necessary.  Any basis for any subalgebra of $\Dist{G}$ that consists of vectors of the form given specified in (\ref{E:totalweight}) will be referred to as a ``PBW-like basis.''

A PBW-like basis for $\Dist{G_r}$ is then given by all elements of the form as in (\ref{E:totalweight}) where $m_i, n_i < p^r$ for all $i$, and $H$ is an element in a basis of $\Dist{T_r}$.

If $t \in T$, then the adjoint action of $t$ on $X_{\alpha}^{(n)}$ is
$$t.X_{\alpha}^{(n)}=\alpha(t)^nX_{\alpha}^{(n)},$$
so $X_{\alpha}^{(n)}$ is a weight vector for $T$ of weight $n\alpha$.  The element in (\ref{E:totalweight}) is then a weight vector of weight
$$\sum_{i=1}^s (n_i-m_i)\alpha_i.$$

\subsection{The fixed-point subalgebras}

Let $r \ge 1$, and let $H \le G$ be a subgroup scheme.  The conjugation action of $H$ on $G$ stabilizes $G_r$.  This defines an ``adjoint action'' of $H$ on $\Dist{G_r}$.  In particular, under this action the multiplication map
$$m: \Dist{G_r} \otimes \Dist{G_r} \rightarrow \Dist{G_r},$$
is a homomorphism of $H$-modules.  The fixed-point submodule $\Dist{G_r}^H$ is then a subalgebra of $\Dist{G_r}$.

If $M$ is a $G$-module, or more generally a $G_rH$-module (where $G_rH$ is the subgroup scheme product of $G_r$ and $H$), then the linear map
$$a_M: \Dist{G_r} \otimes M \rightarrow M,$$
that defines the $G_r$-module structure is also homomorphism of $H$-modules, where $H$ is acting on $\Dist{G_r}$ via the adjoint action and on $M$ via restriction from $G$.

\subsection{Weights of $T$-modules}\label{SS:weights}

Let $\mathbb{X}$ denote the character group (or weight lattice) of $T$, $\mathbb{X}^+$ the set of dominant weights, and $\mathbb{X}_r$ the set of $p^r$-restricted weights.  If $M$ is a $T$-module and $\mu \in \mathbb{X}$, then the $\mu$-weight space of $M$ will be denoted by $M_{\mu}$.  The set of $T$-fixed points of $M$ corresponds to the $0$-weight space, so $M^T = M_0$.

\subsection{$G_rT$-modules and $G_r$}

For each $\lambda \in \mathbb{X}$ there is a simple $G_rT$-module $\widehat{L}_r(\lambda)$ of highest weight $\lambda$.  Every simple $G_rT$-module is of this form, and we have
$$\widehat{L}_r(\lambda+p\lambda^{\prime}) \cong \widehat{L}_r(\lambda) \otimes \Bbbk_{p\lambda^{\prime}}$$
where for a weight $\sigma$, $\Bbbk_{\sigma}$ denotes the one-dimensional $T$-module of that weight.
The simple $G_rT$-module $\widehat{L}_r(\lambda)$ restricts to the simple $G_r$-module $L_r(\lambda)$, and we have that
$$L_r(\lambda) \cong L_r(\mu) \iff \lambda-\mu \in p^r\mathbb{X}.$$

\section{Algebraic Structure of $\Dist{G_r}^T$}

\subsection{Basis and Dimension}

We start with a description of a basis for $\Dist{G_r}^T$.  As noted in Subsection \ref{SS:weights}, this is the same as finding a basis for the weight space $\Dist{G_r}_0$.  Every element in the PBW-like basis for $\Dist{G_r}$ is a weight vector for $T$.  Thus for any $\mu$, a basis for $\Dist{G_r}_{\mu}$ can be given by the PBW-like basis elements having that weight.

We observed earlier that the weight of the element
$$H \cdot X_{-\alpha_1}^{(m_1)}\cdots X_{-\alpha_s}^{(m_s)} X_{\alpha_1}^{(n_1)}\cdots X_{\alpha_s}^{(n_s)}$$
is
$$\sum_{i=1}^s (n_i-m_i)\alpha_i = \sum_{i=1}^s n_i\alpha_i - \sum_{i=1}^s m_i\alpha_i$$
Clearly if $n_i=m_i$ for all $i$, then we get a weight $0$ element.  But this happens more generally whenever
$$\sum_{i=1}^s n_i\alpha_i = \sum_{i=1}^s m_i\alpha_i.$$
For instance, let $G=SL_3$, with $\alpha_1$ and $\alpha_2$ denoting the simple roots, and $\alpha_3=\alpha_1+\alpha_2$ the remaining positive root.  We then have that the element
$$X_{-\alpha_1}X_{-\alpha_2}X_{\alpha_3}$$
has weight $0$.  Our goal then is to describe all elements having this form.

We note that both $\Dist{U_r}$ and $\Dist{U_r^+}$ are $T$-modules.  It is clear that for all weights $\mu$ we have
$$\dim \Dist{U_r}_{-\mu} = \dim \Dist{U_r^+}_{\mu}.$$
The element
$$H \cdot X_{-\alpha_1}^{(m_1)}\cdots X_{-\alpha_s}^{(m_s)} X_{\alpha_1}^{(n_1)}\cdots X_{\alpha_s}^{(n_s)}$$
has weight $0$ if and only if there is some $\mu$ such that
$$X_{\alpha_1}^{(n_1)}\cdots X_{\alpha_s}^{(n_s)} \in  \Dist{U_r^+}_{\mu} \quad \text{and} \quad X_{-\alpha_1}^{(m_1)}\cdots X_{-\alpha_s}^{(m_s)} \in  \Dist{U_r}_{-\mu}.$$
The element $H$ can be any element in $\Dist{T_r}$.

We notice something further.  The adjoint action of $T$ on $\Dist{G_r}$ restricts to the adjoint action of $T_r$ on $\Dist{G_r}$.  This action is (by definition) trivial on $\Dist{G_r}^T$, so that $\Dist{T_r} \le Z(\Dist{G_r}^T)$.  Thus $\Dist{G_r}^T$ is a $\Dist{T_r}$-algebra.  The description of basis elements above now establishes the following.

\begin{prop}
The subalgebra $\Dist{G_r}^{T}$ is a free $\Dist{T_r}$-module of rank
$$\sum_{\mu \in \mathbb{Z}_{\ge 0} \Phi^+} (\dim_{\Bbbk} \Dist{U_r^+}_{\mu})^2.$$
Its dimension as a vector space over $\Bbbk$ is then
$$\left(\sum_{\mu \in \mathbb{Z}_{\ge 0} \Phi^+} (\dim_{\Bbbk} \Dist{U_r^+}_{\mu})^2\right)\cdot (\dim_{\Bbbk} \Dist{T_r}).$$
\end{prop}

\subsection{Augmented as a $\Dist{T_r}$-algebra}

In this subsection we show that $\Dist{G_r}^T$ is augmented as a $\Dist{T_r}$-algebra.  For a subgroup scheme $H \le G$, we let $\Distp{H}$ denote the augmentation ideal of the Hopf algebra $\Dist{H}$.

\begin{lemma}
$\Dist{G_r}$ is a right $\Dist{U_r^+}$-module under multiplication.  Under this action, the submodule $\Dist{G_r}.\Distp{U_r^+}$ is given by the vector subspace with basis consisting of all elements of the form
$$H \cdot X_{-\alpha_1}^{(m_1)}\cdots X_{-\alpha_s}^{(m_s)} X_{\alpha_1}^{(n_1)}\cdots X_{\alpha_s}^{(n_s)}$$
such that $0 \le m_i,n_i < p^r$, and at least one $n_i > 0$.
\end{lemma}

\begin{proof}
This will follow from showing that it holds for a PBW-like basis element of $\Dist{G_r}$ being multiplied by a PBW-like basis element of $\Distp{U_r^+}$ .  An element of the latter type has the form
$$X_{\alpha_1}^{(k_1)}\cdots X_{\alpha_s}^{(k_s)}$$
where at least some $k_i > 0$.  Thus a product as above has the form
$$\left(H \cdot X_{-\alpha_1}^{(m_1)}\cdots X_{-\alpha_s}^{(m_s)} X_{\alpha_1}^{(n_1)}\cdots X_{\alpha_s}^{(n_s)}\right)\left(X_{\alpha_1}^{(k_1)}\cdots X_{\alpha_s}^{(k_s)}\right).$$
This can be expressed as
$$\left(H \cdot X_{-\alpha_1}^{(m_1)}\cdots X_{-\alpha_s}^{(m_s)}\right)\left(X_{\alpha_1}^{(n_1)}\cdots X_{\alpha_s}^{(n_s)}\right)\left(X_{\alpha_1}^{(k_1)}\cdots X_{\alpha_s}^{(k_s)}\right).$$
The second two factors in this product are elements in $\Dist{U_r^+}$, and as the last factor is in the two-sided ideal $\Distp{U_r^+}$, their product must be also.  Thus we have that
$$\left(X_{\alpha_1}^{(n_1)}\cdots X_{\alpha_s}^{(n_s)}\right)\left(X_{\alpha_1}^{(k_1)}\cdots X_{\alpha_s}^{(k_s)}\right)$$
can be expressed as a linear combination of terms
$$X_{\alpha_1}^{(\ell_1)}\cdots X_{\alpha_s}^{(\ell_s)}$$
each with some $\ell_i >0$.  Substituting these in above then yields the claim.
\end{proof}

\begin{theorem}
$\Dist{G_r}^T$ is an augmented $\Dist{T_r}$-algebra.  Specifically, the set of all elements in the PBW-like basis of $\Dist{G_r}^T$ that are not elements in $\Dist{T_r}$ span a two-sided ideal $J \subseteq \Dist{G_r}^T$.  We have a sequence of algebra homomorphisms
$$\Dist{T_r} \rightarrow \Dist{G_r}^T \rightarrow \Dist{G_r}^T/J \cong \Dist{T_r}$$
which is the identity on $\Dist{T_r}$.
\end{theorem}

\begin{proof}
Let $J$ be the span of all elements in the PBW-like basis of $\Dist{G_r}^T$ that are not contained in $\Dist{T_r}$.  We note that if
$$H \cdot X_{-\alpha_1}^{(m_1)}\cdots X_{-\alpha_s}^{(m_s)} X_{\alpha_1}^{(n_1)}\cdots X_{\alpha_s}^{(n_s)}$$
is a basis element in $\Dist{G_r}^T$, then it is not in $\Dist{T_r}$ precisely when at least one $n_i > 0$.  We therefore have that
\begin{equation}\label{E:characterization}
J = \Dist{G_r}^T \cap \Dist{G_r}.\Distp{U_r^+}.
\end{equation}
Now consider another basis element of $\Dist{G_r}^T$ of the form
$$H^{\prime}  \cdot X_{-\alpha_1}^{(m^{\prime}_1)}\cdots X_{-\alpha_s}^{(m^{\prime}_s)} X_{\alpha_1}^{(n^{\prime}_1)}\cdots X_{\alpha_s}^{(n^{\prime}_s)}.$$
If all $n^{\prime}_i=0$, then this element is simply $H^{\prime}$.  Since $H^{\prime}$ commutes with $\Dist{G_r}^T$, we have
$$H \cdot X_{-\alpha_1}^{(m_1)}\cdots X_{-\alpha_s}^{(m_s)} X_{\alpha_1}^{(n_1)}\cdots X_{\alpha_s}^{(n_s)} \cdot H^{\prime} = H\cdot H^{\prime} \cdot X_{-\alpha_1}^{(m_1)}\cdots X_{-\alpha_s}^{(m_s)} X_{\alpha_1}^{(n_1)}\cdots X_{\alpha_s}^{(n_s)},$$
which is an element in $\Dist{G_r}^T \cap \Dist{G_r}.\Distp{U_r^+}$.  On the other hand, if at least some $n^{\prime}_i > 0$, then we have that both the elements
$$\left(H^{\prime}  \cdot X_{-\alpha_1}^{(m^{\prime}_1)}\cdots X_{-\alpha_s}^{(m^{\prime}_s)} X_{\alpha_1}^{(n^{\prime}_1)}\cdots X_{\alpha_s}^{(n^{\prime}_s)}\right) \cdot \left(H \cdot X_{-\alpha_1}^{(m_1)}\cdots X_{-\alpha_s}^{(m_s)} X_{\alpha_1}^{(n_1)}\cdots X_{\alpha_s}^{(n_s)}\right)
$$
$$= \left(H^{\prime}  \cdot X_{-\alpha_1}^{(m^{\prime}_1)}\cdots X_{-\alpha_s}^{(m^{\prime}_s)} X_{\alpha_1}^{(n^{\prime}_1)}\cdots X_{\alpha_s}^{(n^{\prime}_s)} \cdot H \cdot X_{-\alpha_1}^{(m_1)}\cdots X_{-\alpha_s}^{(m_s)}\right) \left(X_{\alpha_1}^{(n_1)}\cdots X_{\alpha_s}^{(n_s)}\right)$$
and
$$\left(H \cdot X_{-\alpha_1}^{(m_1)}\cdots X_{-\alpha_s}^{(m_s)} X_{\alpha_1}^{(n_1)}\cdots X_{\alpha_s}^{(n_s)}\right) \cdot \left(H^{\prime}  \cdot X_{-\alpha_1}^{(m^{\prime}_1)}\cdots X_{-\alpha_s}^{(m^{\prime}_s)} X_{\alpha_1}^{(n^{\prime}_1)}\cdots X_{\alpha_s}^{(n^{\prime}_s)}\right)
$$
$$=  \left(H \cdot X_{-\alpha_1}^{(m_1)}\cdots X_{-\alpha_s}^{(m_s)} X_{\alpha_1}^{(n_1)}\cdots X_{\alpha_s}^{(n_s)} \cdot H^{\prime}  \cdot X_{-\alpha_1}^{(m^{\prime}_1)}\cdots X_{-\alpha_s}^{(m^{\prime}_s)}\right) \left(X_{\alpha_1}^{(n^{\prime}_1)}\cdots X_{\alpha_s}^{(n^{\prime}_s)}\right)$$
are elements in $\Dist{G_r}.\Distp{U_r^+}$, hence are in $J$ by (\ref{E:characterization}).

Having this result above on basis elements, it easily follows that $J$ is closed under left and right multiplication by $\Dist{G_r}^T$.
\end{proof}

\subsection{Compatability with Levi And Other Subgroups}

Suppose that $H \le G$ is a Levi subgroup, or more generally any $W$-conjugate of a Levi subgroup (so is generated by $T$ and a collection of root $SL_2$ subgroups).  Then from the PBW-like basis for $\Dist{G_r}^T$ is it clear that we have
$$\Dist{H_r}^T \subseteq \Dist{G_r}^T.$$

\subsection{$\Dist{G_r}^T$ is noncommutative in general}

\begin{prop}
$\Dist{G_r}^T$ is commutative if and only if $|\Phi^+| = 1$.
\end{prop}

\begin{proof}
If $|\Phi^+| = 1$, then $\Dist{G_r}^T$ is commutative, this is noted by Yoshii in \cite{Y2}.  Suppose now that $|\Phi^+| > 1$.  In view of the result about about Levi subgroups, it suffices to prove noncommutativity for $|\Phi^+| = 2$. In this case there are two simple roots $\alpha_1$ and $\alpha_2$, and we will set $\alpha_3=\alpha_1+\alpha_2$ to be the next positive root in our ordering of $\Phi^+$.  We will show that $X_{-\alpha_1}X_{\alpha_1}$ and $X_{-\alpha_2}X_{\alpha_2}$ do not commute.

Because $\alpha_1$ and $\alpha_2$ are simple roots, the weights $\alpha_1-\alpha_2$ and $\alpha_2-\alpha_1$ are not weights of $\text{Lie}(G)$, hence $X_{-\alpha_1}$ commutes with $X_{\alpha_2}$, and that $X_{\alpha_1}$ commutes with $X_{-\alpha_2}$.  We may also assume that we have fixed our root homomorphisms so that $X_{\alpha_2}X_{\alpha_1}=X_{\alpha_2}X_{\alpha_1}-X_{\alpha_3}$, and similarly for the negative roots.  We now have
$$\left(X_{-\alpha_1}X_{\alpha_1}\right)\left(X_{-\alpha_2}X_{\alpha_2}\right) = X_{-\alpha_1}X_{-\alpha_2}X_{\alpha_1}X_{\alpha_2},$$
while
\begin{align*}
\left(X_{-\alpha_2}X_{\alpha_2}\right)\left(X_{-\alpha_1}X_{\alpha_1}\right) & = X_{-\alpha_2}X_{-\alpha_1}X_{\alpha_2}X_{\alpha_1} \\
& = \left(X_{-\alpha_1}X_{-\alpha_2}-X_{-\alpha_3}\right)\left(X_{\alpha_1}X_{\alpha_2}- X_{\alpha_3} \right).
\end{align*}
This last expression expands out to be
$$X_{-\alpha_1}X_{-\alpha_2}X_{\alpha_1}X_{\alpha_2} -  X_{-\alpha_1}X_{-\alpha_2}X_{\alpha_3} - X_{-\alpha_3}X_{\alpha_1}X_{\alpha_2} + X_{-\alpha_3}X_{\alpha_3}.$$
Thus we see that the elements $X_{-\alpha_1}X_{\alpha_1}$ and $X_{-\alpha_2}X_{\alpha_2}$ do not commute.
\end{proof}

\subsection{$\Dist{G_r}^T$ is not a Hopf subalgebra}

Let $\alpha$ be a root, and consider the element $X_{\alpha}X_{-\alpha} \in \Dist{G_r}^T$.  Both elements, being in $\text{Lie}(G)$, have primitive comultiplication.  Thus the comultiplication of $X_{-\alpha}X_{\alpha}$ is
\begin{align*}
\Delta(X_{-\alpha}X_{\alpha}) & = (X_{-\alpha} \otimes 1 + 1 \otimes X_{-\alpha}) (X_{\alpha} \otimes 1 + 1 \otimes X_{\alpha}) \\
& = X_{-\alpha}X_{\alpha} \otimes 1 + X_{-\alpha} \otimes X_{\alpha} + X_{\alpha} \otimes X_{-\alpha} + 1 \otimes X_{-\alpha}X_{\alpha}.\\
\end{align*}
We thus see that this element is not contained inside of $\Dist{G_r}^T \otimes \Dist{G_r}^T$.

\subsection{Algebra automorphisms and anti-automorphisms}

Although $\Dist{G_r}^T$ is not a Hopf algebra, the antipode $\eta$ of $\Dist{G}$ restricts to the antipode on $\Dist{H}$ for any subgroup scheme $H$.  So $\eta$ maps elements in $\Dist{T_r}$ to $\Dist{T_r}$, and on each $X_{\alpha}^{(n)}$ we have $\eta(X_{\alpha}^{(n)})=(-1)^nX_{\alpha}^{(n)}$ (this can be seen by noting that $\Dist{U_{\alpha}} \cong \Dist{\mathbb{G}_a}$, and the antipode on the latter arises from the group map sending $s \in \mathbb{G}_a$ to $-s$).  Thus $\eta$ restricts to an anti-automorphism of $\Dist{G_r}^T$.  Consequently, given a (left) $\Dist{G_r}^T$ -module $M$, its dual space $M^*$ can be given the structure of a left $\Dist{G_r}^T$ -module via $\eta$.

Additionally,  the algebraic group $G$ has another anti-automorphism $\tau$ which swaps the positive and negative root subgroups and fixes $T$.  In the case of $GL_n$ or $SL_n$, this $\tau$ anti-automorphism can be given by the transpose map on matrices, see \cite[II.1.16]{rags}.  The anti-automorphism $d\tau$ also stabilizes $\Dist{G_r}^T$ , so we can speak of the $\tau$-twisted module $^{\tau}M$ as we can for $G$-modules.  We emphasize here that we are not assuming that $M$ is coming from a $\Dist{G_r}$-module, only that it is a $\Dist{G_r}^T$-module, for otherwise these observations would be triivial.

Finally, we note that we may compose $\eta$ and $d\tau$ to obtain an algebra automorphism of $\Dist{G_r}^T$ .  Further automorphisms can be given by the action of the Weyl group $W$ on $\Dist{G_r}^T$ (in certain root types, the previous automorphism will be realized by the automorphism corresponding to the longest element in $W$).  We note that since these algebras do not fix $T_r$, they will be $\Bbbk$-algebra automorphisms of $\Dist{G_r}^T$ only, but not $\Dist{T_r}$-algebra automorphisms.

\subsection{Algebra generators}

In the next section we show that there is a close connection between simple $\Dist{G_r}^T$-modules and simple $G_rT$-modules.  In studying the former, it would be useful to identify an efficient list of algebra generators for $\Dist{G_r}^T$.

In the case of $G=SL_2$, the elements $1$, along with $X_{-\alpha}^{(p^n)}X_{\alpha}^{(p^n)}$ for $n \ge 0$, generate $\Dist{G_r}^T$ as a $\Dist{T_r}$-algebra.  In particular, for $r=1$ we only need the elements $1$ and $X_{-\alpha}X_{\alpha}$.  The general situation is much more complicated, leading us to ask the following question, posed for simplicity in the $r=1$ case:

\begin{quest}\label{Q:generators}
What is the minimal set of generators needed for the algebra  $\Dist{G_1}^T$?
\end{quest}

\section{Representation Theory of $\Dist{G_r}^T$}

\subsection{Simple $\Dist{T_r}$-modules and Central Characters}

We recall that the simple $\Dist{T_r}$-modules are all one-dimensional, and are in bijection with the character group of $T_r$.  This latter set is isomorphic to $\mathbb{X}/p^r\mathbb{X}$, which in turn is in bijection with $\mathbb{X}_r$.  For $\lambda \in \mathbb{X}$ we may thus speak of the simple $\Dist{T_r}$-module $\Bbbk_{\lambda}$,understanding that $\Bbbk_{\lambda} \cong \Bbbk_{\mu}$ if and only if $\lambda-\mu \in p^r\mathbb{X}$.  Since $\Dist{G_r}^T$ is an augmented $\Dist{T_r}$-algebra, every simple $\Dist{T_r}$-module pulls back via the augmentation to a one-dimensional $\Dist{G_r}^T$-module.  We will also denote by $\Bbbk_{\lambda}$ the pullback of $\Bbbk_{\lambda}$.  We note that there may be more than one $\Dist{G_r}^T$-module that is isomorphic to $\Bbbk_{\lambda}$ when restricted to
$$\Dist{T_r} \subseteq \Dist{G_r}^T.$$
However, $\Bbbk_{\lambda}$ is the unique one in which the elements in the augmentation ideal $J$ all act as $0$.

Since $\Dist{T_r}$ is central in $\Dist{G_r}^T$, every $\Dist{G_r}^T$-module $M$ decomposes as a direct sum according to the characters of $T_r$.  Suppose now that $M$ actually comes via restriction from a $G_rT$-module.  Then the action map
$$a_M:\Dist{G_r} \otimes M \rightarrow M$$
is a $T$-module homomorphism, where $T$ is acting via conjugation on $\Dist{G_r}$.  The action of the subalgebra $\Dist{G_r}^T$ on $M$ preserves the $T$-weight spaces.  That is (the restriction of) the action map gives mappings
$$a_M:\Dist{G_r}^T \otimes M_{\mu} \rightarrow M_{\mu}.$$
Thus each such $M$ decomposes over $\Dist{G_r}^T$ according to these weight spaces.

\subsection{Simple $\Dist{G_r}^T$-modules}

In the previous subsection we saw that every simple $\Dist{T_r}$-module can be pulled back to a simple $\Dist{G_r}^T$-module.  While it is ultimately desireable to develop an intrinsic classification of the simple $\Dist{G_r}^T$-modules, here we will give an extrinsic argument by utilizing the simple $G_rT$-modules.

\begin{prop}\label{P:semisimplerestriction}
If $M$ is a simple $G_rT$-module, then over $\Dist{G_r}^T$ we have
$$M \cong \bigoplus_{\mu \in \mathbb{X}} M_{\mu},$$
and each $\{0\} \ne M_{\mu}$ is a simple $\Dist{G_r}^T$-module.
\end{prop}

\begin{proof}
The decomposition statement is clear, but it remains to show that these factors are simple modules.  Let $M_{\mu}$ be a non-zero weight space, and let $0 \ne v \in M_{\mu}$.  We will show that $\Dist{G_r}^T.v = M_{\mu}$.

Since $M$ is simple as a $G_rT$-module, it is also simple as a $G_r$-module, hence is generated over $\Dist{G_r}$ by the non-zero element $v$.  In particular, for each element $w \in M_{\mu}$ there is some $x \in \Dist{G_r}$ such that $x.v = w$.  We can write $x$ as a linear combination of weight vectors for $T$ of distinct weights.  If any factor has weight non-zero, then it must act as $0$ on $v$ since $x.v$ is a weight vector of the same weight as $v$.  Thus we may assume after removing unnecessary terms that $x$ consists only of weight zero vectors, or in other words that $x \in \Dist{G_r}^T$.  Thus $M_{\mu}$ is a simple $\Dist{G_r}^T$-module.
\end{proof}

\begin{prop}\label{P:uniquerestriction}
The weight spaces $\widehat{L}_r(\lambda_1)_{\mu_1}$ and $\widehat{L}_r(\lambda_2)_{\mu_2}$ are isomorphic as $\Dist{G_r}^T$-submodules if and only if there is some $\gamma \in \mathbb{X}$ such that
$$\lambda_1-\lambda_2=\mu_1-\mu_2=p^r\gamma.$$
\end{prop}

\begin{proof}
($\Rightarrow$) Suppose that there is an isomorphism
$$f: \widehat{L}_r(\lambda_1)_{\mu_1} \rightarrow \widehat{L}_r(\lambda_2)_{\mu_2}$$
of $\Dist{G_r}^T$-modules.  Because this isomorphism holds as $\Dist{T_r}$-modules, it must be the case that
$$\mu_1 \equiv \mu_2 \; (\text{mod } p^r\mathbb{X}).$$
Thus there is some $\gamma \in \mathbb{X}$ such that $\mu_1-\mu_2=p^r\gamma$.  We will now show that
$$\lambda_1-\mu_1 = \lambda_2-\mu_2,$$
so that we also have $\lambda_1-\lambda_2 = p^r\gamma$.

Let $v_{\lambda_1}$ be a highest weight vector of $\widehat{L}_r(\lambda_1)$, and let $0 \ne v \in \widehat{L}_r(\lambda_1)_{\mu_1}$.  There is an element $x \in \Dist{U_r^+}$ of weight $\lambda_1-\mu_1$ such that $x.v = v_{\lambda_1}$.  Similarly, there is some element $y \in \Dist{U_r}$ of weight $\mu_1-\lambda_1$ such that $y.v_{\lambda_1}=v$.  We now have that $yx \in \Dist{G_r}^T$, and $(yx).v = v$.  Since $f$ is an isomorphism of $\Dist{G_r}^T$-modules, we must have that
$$(yx).f(v)=f((yx).v)=f(v).$$
Now $\widehat{L}_r(\lambda_2)_{\mu_2}$ is a $G_rT$-module, and the action of $yx \in \Dist{G_r}^T$ on it comes by restricting the $\Dist{G_r}$ action, so we have that
$$f(v) = (yx).f(v) = y.(x.f(v)).$$
In particular, we have that $x.f(v) \ne 0$.  Thus we have a nonzero vector
$$x.f(v) \in \widehat{L}_r(\lambda_2)_{\mu_2+\lambda_1-\mu_1}.$$
Since $\lambda_2$ is the highest weight of this module, we get
$$\lambda_2 \ge \mu_2 + \lambda_1 - \mu_1,$$
so that
$$\lambda_2 - \mu_2 \ge \lambda_1 - \mu_1.$$
A similar argument can be used to show that
$$\lambda_1 - \mu_1 \ge \lambda_2 - \mu_2,$$
so that
$$\lambda_1 - \mu_1 = \lambda_2 - \mu_2.$$

\noindent ($\Leftarrow$) This direction is clear since if
$$\lambda_1-\lambda_2=\mu_1-\mu_2=p^r\gamma,$$
then over $\Dist{G_r}^T$ we have an isomorphism
$$\widehat{L}_r(\lambda_1)_{\mu_1} \cong \widehat{L}_r(\lambda_2)_{\mu_2} \otimes \Bbbk_{p^r\gamma}.$$
\end{proof}

\begin{lemma}
Let $A$ be a finite-dimensional $\Bbbk$-algebra, with $B$ a $\Bbbk$-subalgebra of $A$.  Let $\{S_i\}_{i \in \mathcal{I}}$ (resp. $\{T_j\}_{j \in \mathcal{J}}$) denote a complete set, up to isomorphism, of simple $B$-modules (resp. simple $A$-modules).  We have
$$\sum_{i \in \mathcal{I}} \dim S_i \le \sum_{j \in \mathcal{J}} \dim T_j.$$
\end{lemma}

\begin{proof}
The sum of the dimensions of the simple $B$-modules is equal to the number of indecomposable summands of $B$ as a left $B$-module in any given direct sum decomposition, which in turn is equal to the cardinality of a complete set of primitive orthogonal idempotents in $B$ (see, for example, \cite[1.7]{Ben}), and the same holds {\em mutatis mutandis} for $A$.  But a set of orthonal idempotents from $B$ is also such a set within $A$, hence corresponds to summands of $A$ as a left $A$-module that can be further split (if necessary) to a direct sum of indecomposable summands.  Hence the cardinality of a complete set of primitive idempotents for $A$ is greater than or equal to that for $B$, and therefore we similarly have this relationship between the sums of the dimensions of their simple modules.
\end{proof}

\begin{theorem}
The set of all weight spaces
$$\{ \{0\} \ne \widehat{L}_r(\lambda)_{\mu} \mid \lambda \in \mathbb{X}_r, \quad \mu \in \mathbb{X} \}$$
restricts over $\Dist{G_r}^T$ to give a complete set, without redundancy, of the simple $\Dist{G_r}^T$-modules.
\end{theorem}

\begin{proof}
From Proposition \ref{P:uniquerestriction}, we know that all such weight spaces are simple as $\Dist{G_r}^T$-modules.  Furthermore, if $\lambda_1, \lambda_2 \in \mathbb{X}_r$, then if $\lambda_1-\lambda_2 = p^r\gamma$ we must have that $\gamma=0$, so no two weight spaces above are isomorphic over $\Dist{G_r}^T$, hence are distinct.  It only remains then to show that every simple $\Dist{G_r}^T$-module is isomorphic to one of these weight spaces restricted to $\Dist{G_r}^T$.  However, by the previous lemma, the sum of the dimensions of the simple $\Dist{G_r}^T$-modules is less than or equal to the sum of the dimensions of the simple $\Dist{G_r}$-modules.  Since the set of $G_rT$-modules
$$\{ \widehat{L}_r(\lambda)\mid \lambda \in \mathbb{X}_r \}$$
restricts to a complete set, without redundancy, of simple $\Dist{G_r}$-modules, it follows that this is the maximum possible sum of dimensions of the simple $\Dist{G_r}^T$-modules, thus all are accounted for.
\end{proof}

The reasoning above also establishes the following:

\begin{theorem}\label{T:completeset}
There is a complete set of primitive orthogonal idempotents for $\Dist{G_r}$ in $\Dist{G_r}^T$.
\end{theorem}

\subsection{Multiplicities in $G_rT$-modules and $\Dist{G_r}^T$-modules}

\begin{theorem}
Let $\mu \in \mathbb{X}$, and $M$ a $G_rT$-module.  Then there is an equality of composition multiplicities
$$[M:\widehat{L}_r(\mu)]=[M_{\mu}:\Bbbk_{\mu}]$$
where the LHS is a statement about $G_rT$-modules and the RHS is a statement about $\Dist{G_r}^T$-modules.
\end{theorem}

\begin{proof}
Let 
$$\{0\}=M_0 \subseteq M_1 \subseteq M_2 \subseteq \cdots \subseteq M_n = M$$
be a composition series for $M$ over $G_rT$.  Then
$$\{0\}=M_0 \subseteq (M_1)_{\mu} \subseteq (M_2)_{\mu} \subseteq \cdots \subseteq (M_n)_{\mu} = M_{\mu}$$
is a $\Dist{G_r}^T$-filtration of $M_{\mu}$.  For each $i$, we have that
$$M_{i+1}/M_i$$
is a simple $G_rT$-module, hence
$$(M_{i+1}/M_i)_{\mu} \cong (M_{i+1})_{\mu}/(M_i)_{\mu}$$
is either $0$ or else is a simple $\Dist{G_r}^T$-module by Proposition \ref{P:semisimplerestriction}.  But $\Bbbk_{\mu}$ appears in this latter filtration precisely when $M_{i+1}/M_i \cong \widehat{L}_r(\mu)$.
\end{proof}

\subsection{Future directions}\label{SS:future}

We expect that the work in this paper will be useful in computing the characters of the simple $G_rT$-modules (and therefore the simple $G$-modules) as follows.  First, determining these characters is equivalent to finding the composition multiplicities of the baby Verma modules $\widehat{Z}_r(\lambda)$ for all $\lambda \in \mathbb{X}_r$ (see, for example, the end of \cite[\S 4]{So}).  As with ordinary Verma modules, the structure of a baby Verma module has an explicit description as a module over $\Dist{G_r}$.  In particular, for any $\mu \in \mathbb{X}$ the vectors in the $\mu$-weight space of $\widehat{Z}_r(\lambda)$ can be easily described.  Using the previous result, one may now compute the composition multiplicity of $\widehat{L}_r(\mu)$ in $\widehat{Z}_r(\lambda)$ simply by looking at the $\Dist{G_r}^T$-module $\widehat{Z}_r(\lambda)_{\mu}$.

For example, if $p \ge h$, the Coxeter number of $G$, then whenever Lusztig's Character Formula holds for $G$ there will be predicted stabilizing of the $G_1$-composition multiplicity of $L_1(\lambda)$ in $Z_1(\lambda)$ for $\lambda$ a $p$-regular weight in a specified $p$-restricted alcove (so the specific weight $\lambda$ will in general change as $p$ does).  This multiplicity can be computed by looking for all $\Dist{G_1}^T$-composition factors of the form $\Bbbk_{\lambda+p\mu}$ in $\widehat{Z}_1(\lambda)_{\lambda+p\mu}$.  Of particular interest would be to compute this number for $\lambda=0$.

The tradeoff in the approach outlined above is that while $\Dist{G_r}^T$ is smaller in dimension than $\Dist{G_r}$, it has an undetermined minimal set of algebra generators.  Indeed, this is the reason that Question \ref{Q:generators} is worth settling.  For example, the algebra $\Dist{G_1}$ is generated as an algebra by a small set of elements, namely all
$$\{ X_{\pm \alpha}, \alpha \in \Pi \}.$$
It would be nice if $\Dist{G_1}^T$ were generated by a set of comparable size, but this has not yet been established.


\providecommand{\bysame}{\leavevmode\hbox
to3em{\hrulefill}\thinspace}

\end{document}